\documentclass[11pt]{article}
\usepackage{amsfonts, amsmath, amsthm, amssymb, latexsym, amscd, color}
\usepackage{tabularx,ragged2e,booktabs}
\usepackage{rotating}
\usepackage{mathrsfs,calrsfs}
\usepackage{textcomp}
\usepackage{graphicx,tikz}
\newcommand{\diam}{\mathrm{diam}}
\newcommand{\dis}{\mathrm{d}}
\newcommand{\diss}{\mathrm{d_2}}

\newcommand{\comment}[1]{}
\newif\ifpdf
\ifpdf \else \fi \textwidth = 6.5 in \textheight = 9 in
\newtheorem{thm}{Theorem}[section]

\newtheorem{observation}[thm]{Observation}

\newtheorem{example}[thm]{Example}
\newtheorem{corollary}[thm]{Corollary}
\newtheorem{lemma}[thm]{Lemma}

\newtheorem{theorem}[thm]{Theorem}
\newtheorem{conjecture}[thm]{Conjecture}

\begin{document}

\title{Diameter of 2-distance graphs}
\author{{\small S.H. Jafari$^{a}$, S.R. Musawi}$^{b}$\\
\\{\small $^{a}$Faculty of Mathematical Sciences, Shahrood University of Technology,}\\ {\small  P.O. Box 36199-95161, Shahrood, Iran}\\
{\small Email: shjafari@shahroodut.ac.ir, shjafari55@gmail.com}\\
\\{\small $^{b}$Faculty of Mathematical Sciences, Shahrood University of Technology,}\\ {\small  P.O. Box 36199-95161, Shahrood, Iran}\\
{\small Email: r\_musawi@shahroodut.ac.ir, r\_musawi@yahoo.com}\\
}
\date{}
\maketitle

\begin{abstract}
For a simple graph $G$, the $2$-distance graph, $D_2(G)$, is a graph with the vertex set $V(G)$ and two vertices are adjacent if and only if their distance is $2$ in the graph $G$. In this paper, for graphs $G$ with diameter 2, we show that $diam(D_2(G))$ can be any integer $t\geqslant2$. For graphs $G$ with $diam(G)\geqslant3$, we prove that $\frac{1}{2}diam(G)\leqslant diam(D_2(G))$ and this inequality is sharp. Also, for $diam(G)=3$, we prove that $diam(D_2(G))\leqslant5$ and this inequality is sharp.
\end{abstract}


\section{introduction}

In this paper, we only consider the undirected finite simple graphs $G = (V,E)$.
 The distance between two vertices of $x,y\in V(G)$, $\dis(x,y)=\dis_G(x,y)$, is the length of the shortest path between them. The diameter of $G$, $\diam(G)$, is the maximum distance between vertices of $G$.
$2$-distance graph of a graph $G$, $D_2(G)$, is a graph with the  vertex set $V(G)$ and two vertices are adjacent if and only their distance is $2$ in the graph $G$. For any $u,v\in V(D_2(G))$, we set $\diss(u,v)=\dis_{D_2(G)}(u,v)$.

The idea of 2-distance graphs is a particular case of a general notion of k-distance
graphs, which was first studied by Harary et al. \cite{hhk}. They investigated the connectedness of 2-distance graph of some graphs. In the book by Prisner  \cite{p}, the dynamics of the k-distance operator was explored. Furthermore, Boland et al. \cite{bhl} extended the k-distance operator to a graph invariant they call distance-n domination number.
To solve one problem posed by Prisner \cite{p}, Zelinka
\cite{z} constructed a graph that is r-periodic under the 2-distance operator for each
positive even integer r. Prisner’s problem was completely solved by Leonor in her
master’s thesis \cite{l}, wherein she worked with a graph construction different from
that of Zelinka’s.
Recently, Azimi and Farrokhi \cite{af2} studied all graphs whose 2-distance graphs
have maximum degree 2. They also solved the problem of finding all graphs whose
2-distance graphs are paths or cycles. More recently, the same authors \cite{af1} determined
all finite simple graphs with no square, diamond, or triangles with a common
vertex that are self 2-distance graphs (that is, graphs $G$ satisfying the equation
$D^2(G)=G$. They further showed the nonexistence of cubic self 2-distance
graphs.
In \cite{c}, Ching   gives some characterizations of 2-distance graphs and finds all graphs X such that $D_2(X)=kP_2$
or $K_m \cup K_n$, where $k \geqslant 2$ is an integer, $P_2$ is the path of order $2$, and $K_m$ is the complete graph of order $m \geqslant 1$. 
In \cite{k} Khormali gives some conditions for connectivity $D_k(G)$. For $D_2(G)$, he proves that if $G$ has no odd cycle, then $D_2(G)$ is disconnected.
In \cite{jm}, the authors characterized all graphs with connected 2-distance graph. Also, for graphs with diameter 2, they proved that $D_2(G)$ is connected if and only if $G$ has no spanning complete bipartite subgraphs.

In this paper, for graphs $G$ with diameter 2, we show that $\diam(D_2(G))$ can be any integer $t\geqslant2$. For graphs $G$ with $\diam(G)\geqslant3$, we prove that $\frac{1}{2}\diam(G)\leqslant \diam(D_2(G))$ and this inequality is sharp. Also, for $\diam(G)=3$ we prove that $\diam(D_2(G))\leqslant5$ and this inequality is sharp.

\section{Main results}  \label{sec2}

In this section, we are going to discuss the diameters of 2-distance graphs. If $G$ is disconnected, then $D_2(G)$ is disconnected too, thus we have to consider  the connected graphs. We are going to find a relation between $\diam(G)$  and $\diam(D_2(G))$, and for this, we have to partition all connected graphs in two distinct partites, the graphs with diameter 2 and graphs with diameter greater than 2.

\subsection{Graphs with diameter 2}
The following observation shows that there is no relation between $\diam(G)$  and $\diam(D_2(G))$.

\begin{observation}\label{obs21}
For any integer $n\geqslant2$ there are some graphs $G$ with 
$\diam(G)=2$ and $\diam(D_2(G))=n$.
\end{observation}

\begin{example}
For  $n=2,3,4$ the following graphs $G_2, G_3,G_4$  satisfies the conditions of Observation \ref{obs21}.
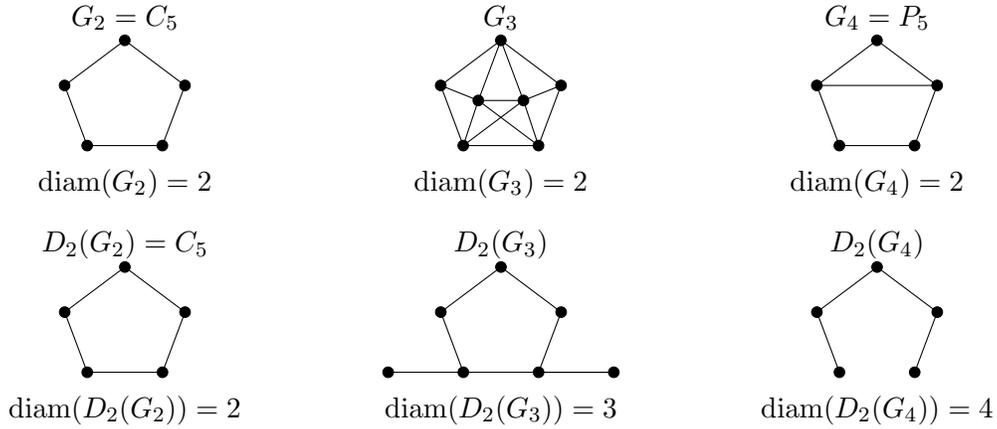
\begin{center}
\begin{tikzpicture}
\node at (0,0){
\begin{tikzpicture}
\draw (0,0)--(1,0)--(1.3,0.8)--(0.5,1.4)--(-.3,.8)--(0,0);
\filldraw (.5,1.4) circle(2pt);\filldraw (1,0) circle(2pt);
\filldraw (0,0) circle(2pt);\filldraw (1.3,0.8) circle(2pt);
\filldraw (-.3,0.8) circle(2pt);
\node at (.5,1.7){$D_2(G_2)=C_5$};
\node at (.5,-.5){$\diam(D_2(G_2))=2$};

\end{tikzpicture}
};
\node at (0,3){
\begin{tikzpicture}
\draw (0,0)--(1,0)--(1.3,0.8)--(0.5,1.4)--(-.3,.8)--(0,0);
\filldraw (.5,1.4) circle(2pt);\filldraw (1,0) circle(2pt);
\filldraw (0,0) circle(2pt);\filldraw (1.3,0.8) circle(2pt);
\filldraw (-.3,0.8) circle(2pt);
\node at (.5,1.7){$G_2=C_5$};
\node at (.5,-.5){$\diam(G_2)=2$};
\end{tikzpicture}
};

\node at (5,0){
\begin{tikzpicture}
\draw (-1,0)--(0,0)--(1,0)--(2,0)--(1,0)--(1.3,0.8)--(0.5,1.4)--(-.3,.8)--(0,0);
\filldraw (.5,1.4) circle(2pt);\filldraw (2,0) circle(2pt);\filldraw (1,0) circle(2pt);
\filldraw (-1,0) circle(2pt);\filldraw (0,0) circle(2pt);\filldraw (1.3,0.8) circle(2pt);
\filldraw (-.3,0.8) circle(2pt);
\node at (.5,1.7){$D_2(G_3)$};
\node at (.5,-.5){$\diam(D_2(G_3))=3$};

\end{tikzpicture}
};
\node at (5,3){
\begin{tikzpicture}
\draw (1.3,0.8)--(0.5,1.4)--(-.3,.8)--(0,0)--(1,0)--(1.3,0.8)--(.8,.6)--(.2,0.6)--(-.3,.8);
\draw (.2,.6)--(0,0)--(.8,.6)--(1,0)--(.2,0.6)--(.5,1.4)--(.8,.6);
\filldraw (.5,1.4) circle(2pt);\filldraw (.8,0.6) circle(2pt);\filldraw (1,0) circle(2pt);
\filldraw (.2,0.6) circle(2pt);\filldraw (0,0) circle(2pt);\filldraw (1.3,0.8) circle(2pt);
\filldraw (-.3,0.8) circle(2pt);
\node at (.5,1.7){$G_3$};
\node at (.5,-.5){$\diam(G_3)=2$};
\end{tikzpicture}
};
\node at (10,0){
\begin{tikzpicture}
\draw (1,0)--(1.3,0.8)--(0.5,1.4)--(-.3,.8)--(0,0);
\filldraw (.5,1.4) circle(2pt);\filldraw (1,0) circle(2pt);
\filldraw (0,0) circle(2pt);\filldraw (1.3,0.8) circle(2pt);
\filldraw (-.3,0.8) circle(2pt);
\node at (.5,1.7){$D_2(G_4)$};
\node at (.5,-.5){$\diam(D_2(G_4))=4$};

\end{tikzpicture}
};
\node at (10,3){
\begin{tikzpicture}
\draw (-.3,.8)--(1.3,0.8)--(0.5,1.4)--(-.3,.8)--(0,0)--(1,0)--(1.3,0.8);
\filldraw (.5,1.4) circle(2pt);\filldraw (1,0) circle(2pt);
\filldraw (0,0) circle(2pt);\filldraw (1.3,0.8) circle(2pt);
\filldraw (-.3,0.8) circle(2pt);
\node at (.5,1.7){$G_4=P_5$};
\node at (.5,-.5){$\diam(G_4)=2$};
\end{tikzpicture}
};
\end{tikzpicture}

\end{center}

\end{example}
\begin{example}

 If $n\geqslant4$ and the graph $G_{n}=P_{n+1}^C$ satisfies the conditions of Observation \ref{obs21}, since $\diam(G)=2$, $D_2(G)=(P_{n+1}^C)^C=P_{n+1}$ and so $\diam(D_2(G))=n$.
\end{example}

\subsection{A lower bound for diameter of 2-distance graphs}

Here, we obtain a lower bound for the diameter of 2-distance graphs of the graphs $G$ with $\diam(G)\geqslant3$.

\begin{lemma}\label{lem1}
Let $G$ be a graph and $D_2(G)$ be connected. For any $u,v\in V(G)$ :
 $$\diss(u,v)\geqslant \left\lceil \frac{1}{2}\dis(u,v)\right\rceil$$
\end{lemma}
\begin{proof}
Let  $\diss(u,v)=k$ and $u=x_0x_1x_2\cdots x_{k}=v$ be a path in $D_2(G)$. Now, there is a walk $u=x_0y_1x_1y_2x_2\cdots y_kx_{k}=v$ in $G$ with length $2k$. Thus $2k\geqslant \dis(u,v)$. Consequently $\diss(u,v)\geqslant \left\lceil \frac{1}{2}\dis(u,v)\right\rceil$.
\end{proof}

An immediate result from Lemma \ref{lem1} is the following.
\begin{theorem}\label{the23}
Suppose that $G$ be a graph and $D_2(G)$ be connected. Then $\diam(D_2(G))\geqslant \left\lceil \frac{1}{2}\diam(G)\right\rceil$
\end{theorem}

The following examples show that the inequality of the above theorem is sharp.

\begin{example}\label{exa26}
In the  following graph $G$, we have $\diam(G)=3$ and $\diam(D_2(G))=2$.
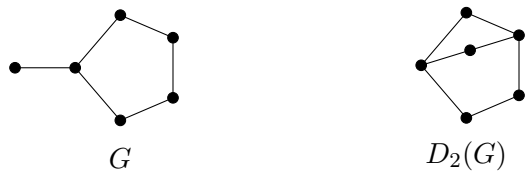
\begin{center}

\begin{tikzpicture}

\node at(0,0){
\begin{tikzpicture}

\draw(-.8,0)--(0,0)--(.6,.7)--(1.3,.4)--(1.3,-.4)--(.6,-.7)--(0,0);
\filldraw (-.8,0) circle(2pt);\filldraw (0,0) circle(2pt);\filldraw (.6,.7) circle(2pt);
\filldraw (.6,-.7) circle(2pt);\filldraw (1.3,.4) circle(2pt);\filldraw (1.3,-.4) circle(2pt);
\node at(.6,-1.2){$G$};
\end{tikzpicture}};
\node at(5,0){
\begin{tikzpicture}

\draw(1.3,.4)--(0,0)--(.6,.7)--(1.3,.4)--(1.3,-.4)--(.6,-.7)--(0,0);
\filldraw (.65,.2) circle(2pt);\filldraw (0,0) circle(2pt);\filldraw (.6,.7) circle(2pt);
\filldraw (.6,-.7) circle(2pt);\filldraw (1.3,.4) circle(2pt);\filldraw (1.3,-.4) circle(2pt);
\node at(.6,-1.2){$D_2(G)$};
\end{tikzpicture}};
\end{tikzpicture}
\end{center}

\end{example}

\begin{example}\label{exa27}
Let $G_{t,3}$ be the  following graph. 

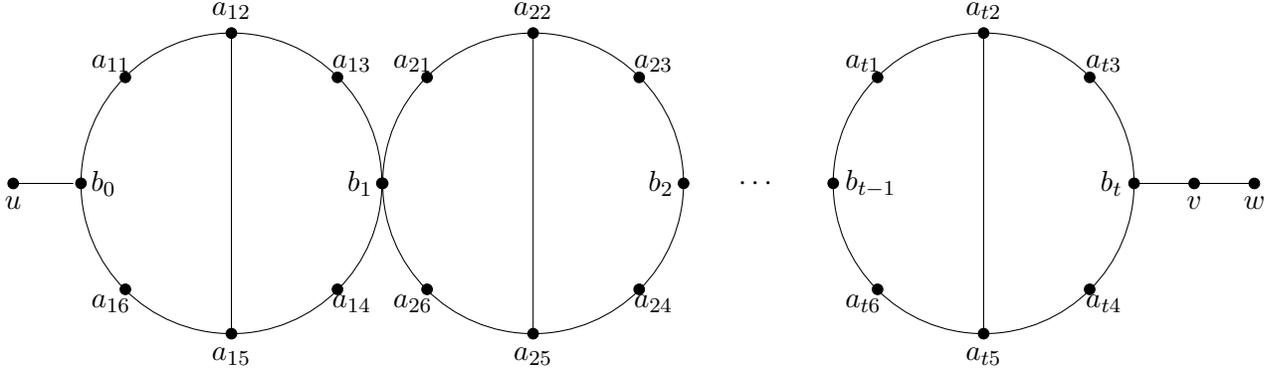
\begin{center}
\begin{tikzpicture}

\draw (-2.1,0)--(-2.9,0);\draw (13.6,0)--(12,0);
\filldraw (-2.9,0) circle(2pt);\filldraw (12.8,0) circle(2pt);\filldraw (13.6,0) circle(2pt);
\node at (-2.9,-.25) {$u$};\node at (12.8,-.25) {$v$};\node at (13.6,-.25) {$w$};
\node at (0,0){
\begin{tikzpicture}
\draw (0,0) circle(20mm);
\draw (0,-2)--(0,2);
\filldraw (-2,0) circle(2pt);\node at (-1.7,0) {$b_0$};
\filldraw (2,0) circle(2pt);
\filldraw (1.41,1.41) circle(2pt);\node at (1.6,1.6) {$a_{13}$};
\filldraw (1.41,-1.41) circle(2pt);\node at (1.6,-1.6) {$a_{14}$};
\filldraw (-1.41,1.41) circle(2pt);\node at (-1.6,1.6) {$a_{11}$};
\filldraw (-1.41,-1.41) circle(2pt);\node at (-1.6,-1.6) {$a_{16}$};
\filldraw (0,2) circle(2pt);\node at (0,2.3) {$a_{12}$};
\filldraw (0,-2) circle(2pt);\node at (0,-2.3) {$a_{15}$};

\end{tikzpicture}
};
\node at (3.75,0){
\begin{tikzpicture}
\draw (0,0) circle(20mm);
\draw (0,-2)--(0,2);
\filldraw (-2,0) circle(2pt);\node at (-2.3,0) {$b_1$};
\filldraw (2,0) circle(2pt);\node at (1.7,0) {$b_2$};
\filldraw (1.41,1.41) circle(2pt);\node at (1.6,1.6) {$a_{23}$};
\filldraw (1.41,-1.41) circle(2pt);\node at (1.6,-1.6) {$a_{24}$};
\filldraw (-1.41,1.41) circle(2pt);\node at (-1.6,1.6) {$a_{21}$};
\filldraw (-1.41,-1.41) circle(2pt);\node at (-1.6,-1.6) {$a_{26}$};
\filldraw (0,2) circle(2pt);\node at (0,2.3) {$a_{22}$};
\filldraw (0,-2) circle(2pt);\node at (0,-2.3) {$a_{25}$};

\end{tikzpicture}
};
\node at (7,0){$\cdots$};
\node at (10,0){
\begin{tikzpicture}
\draw (0,0) circle(20mm);

\draw (0,-2)--(0,2);
\filldraw (-2,0) circle(2pt);\node at (-1.5,0) {$b_{t-1}$};
\filldraw (2,0) circle(2pt);\node at (1.7,0) {$b_{t}$};
\filldraw (1.41,1.41) circle(2pt);\node at (1.6,1.6) {$a_{t3}$};
\filldraw (1.41,-1.41) circle(2pt);\node at (1.6,-1.6) {$a_{t4}$};
\filldraw (-1.41,1.41) circle(2pt);\node at (-1.6,1.6) {$a_{t1}$};
\filldraw (-1.41,-1.41) circle(2pt);\node at (-1.6,-1.6) {$a_{t6}$};
\filldraw (0,2) circle(2pt);\node at (0,2.3) {$a_{t2}$};
\filldraw (0,-2) circle(2pt);\node at (0,-2.3) {$a_{t5}$};

\end{tikzpicture}
};

\end{tikzpicture}
\end{center}

%

Since $b_0,b_1, \cdots, b_{t}$ are cut-vertices, $\diam(G_{t,3})=4t+3$.
Also,  $D_2(G_{t,3})$ is the following.

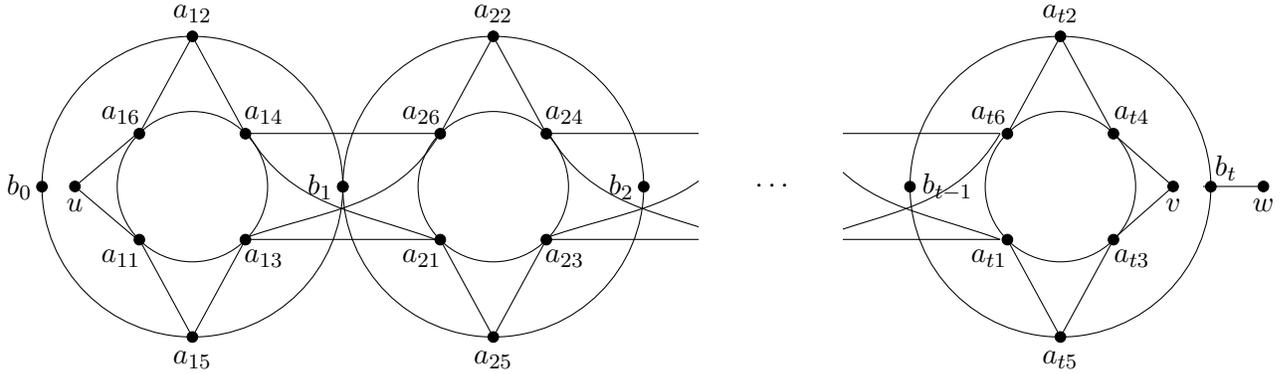
\begin{center}
\begin{tikzpicture}[scale=1]

\draw (-1.3,0)--(-.47,.705);\draw (-1.3,0)--(-.47,-.705);\draw (13.7,0)--(14.5,0);
\draw (13.3,0)--(12.5,.705);\draw (13.3,0)--(12.5,-.705);
\filldraw (-1.3,0) circle(2pt);\filldraw (13.3,0) circle(2pt);\filldraw (14.5,0) circle(2pt);
\node at (-1.3,-.25) {$u$};\node at (13.3,-.25) {$v$};\node at (14.5,-.25) {$w$};

\node at (0,0){
\begin{tikzpicture}
\draw (0,0) circle(20mm);
\draw (0,0) circle(10mm);
\draw (-0.705,0.705)--(0,2);\draw (0.705,0.705)--(0,2);
\draw (-0.705,-0.705)--(0,-2);\draw (0.705,-0.705)--(0,-2);
\filldraw (0.705,0.705) circle(2pt);\node at (.95,.95) {$a_{14}$};
\filldraw(-0.705,0.705) circle(2pt);\node at (.95,-.95) {$a_{13}$};
\filldraw (0.705,-0.705) circle(2pt);\node at (-.95,.95) {$a_{16}$};
\filldraw (-0.705,-0.705) circle(2pt);\node at (-.95,-.95) {$a_{11}$};
\filldraw (-2,0) circle(2pt);\node at (-2.3,0) {$b_0$};
\filldraw (2,0) circle(2pt);

\filldraw (0,2) circle(2pt);\node at (0,2.3) {$a_{12}$};
\filldraw (0,-2) circle(2pt);\node at (0,-2.3) {$a_{15}$};

\end{tikzpicture}
};
\node at (4,0){
\begin{tikzpicture}
\draw (0,0) circle(20mm);
\draw (0,0) circle(10mm);
\draw (-0.705,0.705)--(0,2);\draw (0.705,0.705)--(0,2);
\draw (-0.705,-0.705)--(0,-2);\draw (0.705,-0.705)--(0,-2);
\filldraw (0.705,0.705) circle(2pt);\node at (.95,.95) {$a_{24}$};
\filldraw(-0.705,0.705) circle(2pt);\node at (.95,-.95) {$a_{23}$};
\filldraw (0.705,-0.705) circle(2pt);\node at (-.95,.95) {$a_{26}$};
\filldraw (-0.705,-0.705) circle(2pt);\node at (-.95,-.95) {$a_{21}$};

\filldraw (-2,0) circle(2pt);\node at (-2.3,0) {$b_1$};
\filldraw (2,0) circle(2pt);\node at (1.7,0) {$b_2$};

\filldraw (0,2) circle(2pt);\node at (0,2.3) {$a_{22}$};
\filldraw (0,-2) circle(2pt);\node at (0,-2.3) {$a_{25}$};

\end{tikzpicture}
};
\node at (8,0){$\cdots$};
\node at (12,0){
\begin{tikzpicture}
\draw (0,0) circle(20mm);
\draw (0,0) circle(10mm);
\draw (-0.705,0.705)--(0,2);\draw (0.705,0.705)--(0,2);
\draw (-0.705,-0.705)--(0,-2);\draw (0.705,-0.705)--(0,-2);
\filldraw (0.705,0.705) circle(2pt);\node at (.95,.95) {$a_{t4}$};
\filldraw(-0.705,0.705) circle(2pt);\node at (.95,-.95) {$a_{t3}$};
\filldraw (0.705,-0.705) circle(2pt);\node at (-.95,.95) {$a_{t6}$};
\filldraw (-0.705,-0.705) circle(2pt);\node at (-.95,-.95) {$a_{t1}$};
\filldraw (-2,0) circle(2pt);\node at (-1.5,0) {$b_{t-1}$};
\filldraw (2,0) circle(2pt);\node at (2.2,.25) {$b_{t}$};

\filldraw (0,2) circle(2pt);\node at (0,2.3) {$a_{t2}$};
\filldraw (0,-2) circle(2pt);\node at (0,-2.3) {$a_{t5}$};

\end{tikzpicture}
};

\draw (0.905,0.705)--(3.495,.705);\draw (0.905,-0.705)--(3.495,-.705);
\draw (1.0,0.705)to[out=-60,in=160](3.495,-.705);
\draw (0.905,-0.705)to[out=20,in=-120](3.55,.705);

\draw (5.0,0.705)to[out=-60,in=160](7.495,-.705);
\draw (4.905,-0.705)to[out=20,in=-120](7.55,.705);
\draw (5.0,0.705)--(7.495,.705);
\draw (4.905,-0.705)--(7.55,-.705);

\draw (8.5,0.705)to[out=-60,in=160](11.0,-.705);
\draw (8.5,-0.705)to[out=20,in=-120](11,.705);
\draw (8.5,0.705)--(11.0,.705);
\draw (8.5,-0.705)--(11,-.705);

\filldraw [white](7,1)rectangle(7.7,-1);
\filldraw [white](8.3,1)rectangle(8.9,-1);

\end{tikzpicture}
\end{center}

It is easily seen that $\diss(u,w)=2t+2$ and then  $\diam(D_2(G_{t,3}))=2t+2$.

Also, By setting $G_{t,2}=G_{t,3}-w$, $G_{t,1}=G_{t,3}-\{v,w\}$ and $G_{t,0}=G_{t,3}-\{u,v,w\}$, we have $\diam(G_{t,i})=4t+i$ and $\diam(D_2(G_{t,i}))=2t+\left\lceil \dfrac{i}{2}\right \rceil$.
\end{example}

\subsection{An upper bound for the diameter of 2-distance graphs}

In this subsection, we obtain an upper bound for the diameter of 2-distance graphs of the graphs $G$ with $\diam(G)=3$.

\begin{theorem}\label{the26}
If $\diam(G)=3$ then $\diam(D_2(G))\leqslant 5$.
\end{theorem}

\begin{proof}
We prove this by contradiction.
\\
Suppose that $G$ be a graph with $\diam(G)=3$ and $a_1a_2a_3a_4a_5a_6a_7$ is a path with length $6$ in $D_2(G)$.

\textbf{Claim 1:}  $\dis(a_1,a_3)=1$
\begin{proof}
We  consider the following two possible cases.

\textbf{Case 1:}  $\dis(a_1,a_3)=3 , \dis(a_3,a_7)=1$

Let $a_1bca_3$ be a path in $G$. Since $ca_3a_7$ is a path in G, $ \dis(c,a_7)\leqslant2$. If $ \dis(c,a_7)=2$ then $a_1ca_7$ is a path in $D_2(G)$, a contradiction. Thus $ \dis(c,a_7)=1$.
If $ \dis(b,a_7)=2$ then $a_1a_2a_3ba_7$ is a path in $D_2(G)$, a contradiction. Thus $ \dis(b,a_7)=1$.
Consequently, $\dis(a_1,a_7)\leqslant2$. But, $\dis(a_1,a_7)\ne2$ and then $\dis(a_1,a_7)=1$. Thus, $a_1a_7a_3$  is a path in $G$ and $\dis(a_1,a_3)\leqslant2$, a contradiction.

\textbf{Case 2:}  $\dis(a_1,a_3)=3 , \dis(a_3,a_7)=3$.

Let $a_1bca_3$ and $a_3dea_7$ be  paths in $G$. 
Since, $\dis(a_1,a_3)=3$, $b\ne d$, similarly $c\ne e$. Then we have the following three subcases.

\textbf{Subcase 1:}  $c= d$.

Then $a_1ca_7$ is a path in $D_2(G)$, a contradiction.

\textbf{Subcase 2:} $c\ne d$ and $b=e$.
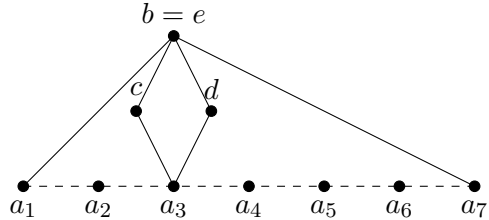
\begin{center}
\begin{tikzpicture}
\foreach \x in {1,2,...,7} \filldraw (\x,0) circle(2pt);
\draw [dashed](1,0)--(7,0);
\draw (1,0)--(3,2)--(7,0);
\draw (3,2)--(2.5,1)--(3,0)--(3.5,1)--(3,2);
\filldraw (3,2) circle(2pt);
  \filldraw (2.5,1) circle(2pt); \filldraw (3.5,1) circle(2pt);
\node at (1,-.3){$a_1$};\node at (2,-.3){$a_2$};\node at (3,-.3){$a_3$};
\node at (4,-.3){$a_4$};\node at (5,-.3){$a_5$};\node at (6,-.3){$a_6$};
\node at (7,-.3){$a_7$};\node at (3,2.3){$b=e$};\node at (2.5,1.3){$c$};
\node at (3.5,1.3){$d$};
\end{tikzpicture}
\end{center}

If $\dis(a_1,d)=2$ then $a_1da_7$ is a path in $D_2(G)$, a contradiction. Thus $\dis(a_1,d)=1$. Consequently, $ \dis(a_1,a_3)\leqslant2$, a contradiction.

\textbf{Subcase 3:} $c\ne d$ and $b\ne e$.
\\
If $\dis(c,d)=2$, then $a_1cda_7$ is a path in $D_2(G)$, a contradiction. Then, we have the following.

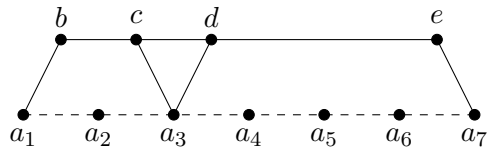
\begin{center}
\begin{tikzpicture}
\foreach \x in {1,2,...,7} \filldraw (\x,0) circle(2pt);
\draw [dashed](1,0)--(7,0);
\draw (1,0)--(1.5,1)--(2.5,1)--(3,0)--(3.5,1)--(6.5,1)--(7,0);
\draw(2.5,1)--(3.5,1);
 \filldraw (1.5,1) circle(2pt); \filldraw (2.5,1) circle(2pt); \filldraw (3.5,1) circle(2pt); \filldraw (6.5,1) circle(2pt);
\node at (1,-.3){$a_1$};\node at (2,-.3){$a_2$};\node at (3,-.3){$a_3$};
\node at (4,-.3){$a_4$};\node at (5,-.3){$a_5$};\node at (6,-.3){$a_6$};
\node at (7,-.3){$a_7$};\node at (1.5,1.3){$b$};\node at (2.5,1.3){$c$};
\node at (3.5,1.3){$d$};\node at (6.5,1.3){$e$};
\end{tikzpicture}
\end{center}

If $ \dis(b,d)=2$ then $a_1a_2a_3bda_7$ is a path in $D_2(G)$, a contradiction. Then $ \dis(b,d)=1$. 
If $ \dis(a_1,d)=2$ then $a_1da_7$ is a path in $D_2(G)$, a contradiction. Then $ \dis(a_1,d)=1$. 
Consequently, $ \dis(a_1,a_3)\leqslant2$, a contradiction. 
\end{proof}  

Therefore, $\dis(a_1,a_3)=1$ and by symmetric $\dis(a_5,a_7)=1$

\textbf{Claim 2:}  $\dis(a_3,a_5)=1$.
\begin{proof}
By contradiction, let $\dis(a_3,a_5)=3$ and $a_3bca_5$ be a path in $G$. Since $\dis(b,a_5)=\dis(c,a_3)=2$, $a_1\ne b,c$. Similarly, $a_7\ne b,c$. Noe, we have the following

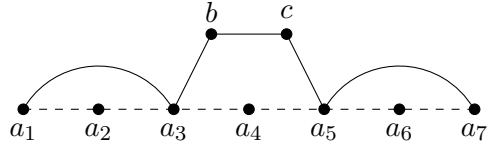
\begin{center}
\begin{tikzpicture}
\foreach \x in {1,2,...,7} \filldraw (\x,0) circle(2pt);
\draw [dashed](1,0)--(7,0);
\draw (3,0)--(3.5,1)--(4.5,1)--(5,0);
 \draw (1,0) arc (150:30:1.15);
 \draw (5,0) arc (150:30:1.15);
 \filldraw (3.5,1) circle(2pt); \filldraw (4.5,1) circle(2pt);
\node at (1,-.3){$a_1$};\node at (2,-.3){$a_2$};\node at (3,-.3){$a_3$};
\node at (4,-.3){$a_4$};\node at (5,-.3){$a_5$};\node at (6,-.3){$a_6$};
\node at (7,-.3){$a_7$};\node at (3.5,1.3){$b$};\node at (4.5,1.3){$c$};
\end{tikzpicture}
\end{center}

If $\dis(a_1,b)=2$ then $a_1ba_5$ is a path in $D_2(G)$, a contradiction. Thus $\dis(a_1,b)=1$ and by symmetry $\dis(c,a_7)=1$.

If $\dis(a_1,c)=1$ then $\dis(a_1,a_5)\leqslant2$ and so $\dis(a_1,a_5)=1$. Consequently $\dis(a_3,a_5)\leqslant2$, a contradiction. Thus $\dis(a_1,c)=2$ and by symmetry $\dis(b,a_7)=2$.

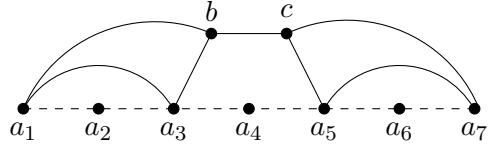
\begin{center}
\begin{tikzpicture}
\foreach \x in {1,2,...,7} \filldraw (\x,0) circle(2pt);
\draw [dashed](1,0)--(7,0);
\draw (3,0)--(3.5,1)--(4.5,1)--(5,0);
 \draw (1,0) arc (150:30:1.15); \draw (1,0) arc (157:65:1.9);
 \draw (5,0) arc (150:30:1.15);\draw (4.5,1) arc (115:23:1.9);
 \filldraw (3.5,1) circle(2pt); \filldraw (4.5,1) circle(2pt);
\node at (1,-.3){$a_1$};\node at (2,-.3){$a_2$};\node at (3,-.3){$a_3$};
\node at (4,-.3){$a_4$};\node at (5,-.3){$a_5$};\node at (6,-.3){$a_6$};
\node at (7,-.3){$a_7$};\node at (3.5,1.3){$b$};\node at (4.5,1.3){$c$};
\end{tikzpicture}
\end{center}

Since $\dis(a_3,a_4)=\dis(a_4,a_5)=2$, $a_4\ne b,c$.

 If $\dis(b,a_4)=2$ then $a_7ba_4$ is a path in $D_2(G)$, a contradiction. Thus $\dis(b,a_4)=1 \text{ or }3$.
 Let $\dis(b,a_4)=3$ and $ba_3ua_4$ be a path in $G$. Since $\dis(b,a_4)=3$, $u\ne a_1,b$. If $\dis(a_1,u)=2$ then $a_1uba_7$ is a path in $D_2(G)$, a contradiction. Thus $\dis(a_1,u)=1$, $\dis(a_1,a_4)\leqslant2$ and so $\dis(a_1,a_4)=1$. Consequently,  $\dis(b,a_4)\leqslant2$, a contradiction. Therefore, $\dis(b,a_4)=1$ and by symmetry $\dis(c,a_4)=1$.

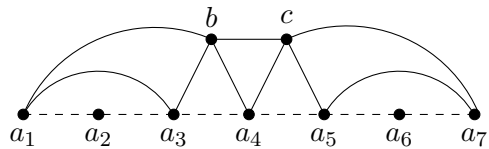
\begin{center}
\begin{tikzpicture}
\foreach \x in {1,2,...,7} \filldraw (\x,0) circle(2pt);
\draw [dashed](1,0)--(7,0);
\draw (3,0)--(3.5,1)--(4.5,1)--(5,0);
\draw (3.5,1)--(4,0)--(4.5,1);
 \draw (1,0) arc (150:30:1.15); \draw (1,0) arc (157:65:1.9);
 \draw (5,0) arc (150:30:1.15);\draw (4.5,1) arc (115:23:1.9);
 \filldraw (3.5,1) circle(2pt); \filldraw (4.5,1) circle(2pt);
\node at (1,-.3){$a_1$};\node at (2,-.3){$a_2$};\node at (3,-.3){$a_3$};
\node at (4,-.3){$a_4$};\node at (5,-.3){$a_5$};\node at (6,-.3){$a_6$};
\node at (7,-.3){$a_7$};\node at (3.5,1.3){$b$};\node at (4.5,1.3){$c$};
\end{tikzpicture}
\end{center}

Now, we have $\dis(a_1,a_4)\leqslant2$ and then $\dis(a_1,a_4)=1$. By symmetry $\dis(a_4,a_7)=1$  and then $\dis(a_1,a_7)\leqslant2$. Thus $\dis(a_1,a_7)=1$ and so $\dis(a_1,a_5)=1$ which shows $\dis(a_3,a_5)\leqslant2$, a contradiction.
\end{proof}  

Therefore, $\dis(a_3,a_5)=1$ and then $\dis(a_1,a_5)=\dis(a_3,a_7)=\dis(a_1,a_7)=1$.

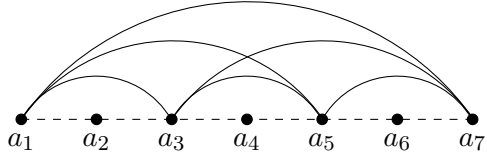
\begin{center}
\begin{tikzpicture}
\foreach \x in {1,2,...,7} \filldraw (\x,0) circle(2pt);
\draw [dashed](1,0)--(7,0);
 \draw (1,0) arc (150:30:1.15); \draw (1,0) arc (145:35:2.45);
 \draw (3,0) arc (150:30:1.15); 
 \draw (5,0) arc (150:30:1.15); \draw (3,0) arc (145:35:2.45);
\draw (1,0) arc (145:35:3.67);
\node at (1,-.3){$a_1$};\node at (2,-.3){$a_2$};\node at (3,-.3){$a_3$};
\node at (4,-.3){$a_4$};\node at (5,-.3){$a_5$};\node at (6,-.3){$a_6$};
\node at (7,-.3){$a_7$};
\end{tikzpicture}
\end{center}

\textbf{Claim 3:}  $\dis(a_2,a_4)=1$

\begin{proof}
By contradiction, let $\dis(a_2,a_4)=3$. 

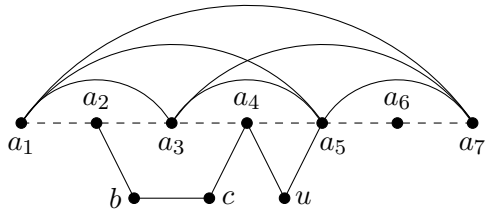
\begin{center}
\begin{tikzpicture}
\foreach \x in {1,2,...,7} \filldraw (\x,0) circle(2pt);
\draw [dashed](1,0)--(7,0);
\draw (2,0)--(2.5,-1)--(3.5,-1)--(4,0)--(4.5,-1)--(5,0);

 \draw (1,0) arc (150:30:1.15); \draw (1,0) arc (145:35:2.45);
 \draw (3,0) arc (150:30:1.15); 
 \draw (5,0) arc (150:30:1.15); \draw (3,0) arc (145:35:2.45);
\draw (1,0) arc (145:35:3.67);
\filldraw (2.5,-1) circle(2pt);\filldraw (3.5,-1) circle(2pt);\filldraw (4.5,-1) circle(2pt);
\node at (1,-.3){$a_1$};\node at (2,.3){$a_2$};\node at (3,-.3){$a_3$};
\node at (4,.3){$a_4$};\node at (5.15,-.3){$a_5$};\node at (6,.3){$a_6$};
\node at (7,-.3){$a_7$};
\node at (2.25,-1){$b$};\node at (3.75,-1){$c$};\node at (4.75,-1){$u$};
\end{tikzpicture}
\end{center}

If $\dis(u,a_7)=1$ then $\dis(a_4,a_7)\leqslant2$. Thus $\dis(a_4,a_7)=1$. If $\dis(c,a_7)=2$ then $a_2ca_7$ is a path in $D_2(G)$, a contradiction, and so $\dis(c,a_7)=1$.
If $\dis(b,a_7)=2$ then $a_4ba_7$  is a path in $D_2(G)$, a contradiction, and so $\dis(b,a_7)=1$.
Thus $\dis(a_2,a_7)\leqslant2$ and $\dis(a_2,a_7)=1$. Which shows $\dis(a_2,a_4)\leqslant2$, a contradiction, and so Then $\dis(u,a_7)=2$.

Now $\dis(u,a_1)\leqslant2$. If $\dis(u,a_1)=2$ then $a_1ua_7$ is a path, a contradiction. Thus $\dis(u,a_1)=1$. Consequently, $\dis(a_1,a_4)=1$. And also $\dis(a_4,a_7)=1$. Similarly, we have a contradiction.
\end{proof}

Thus, $\dis(a_2,a_4)=1$, and by symmetry, $\dis(a_4,a_6)=1$  and then $\dis(a_2,a_6)=1$.

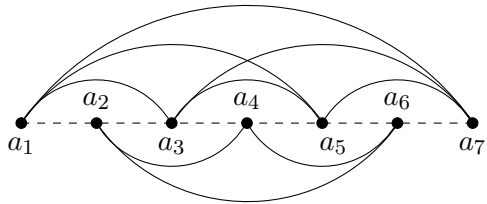
\begin{center}
\begin{tikzpicture}
\foreach \x in {1,2,...,7} \filldraw (\x,0) circle(2pt);
\draw [dashed](1,0)--(7,0);
 \draw (1,0) arc (150:30:1.15); 
\draw (1,0) arc (145:35:2.45);
 \draw (3,0) arc (150:30:1.15); 
 \draw (5,0) arc (150:30:1.15); 
\draw (3,0) arc (145:35:2.45);
\draw (1,0) arc (145:35:3.67);
 \draw (2,0) arc (-150:-30:1.15); 
  \draw (4,0) arc (-150:-30:1.15); 
\draw (2,0) arc (-145:-35:2.45);
\node at (1,-.3){$a_1$};\node at (2,.3){$a_2$};\node at (3,-.3){$a_3$};
\node at (4,.3){$a_4$};\node at (5.15,-.3){$a_5$};\node at (6,.3){$a_6$};
\node at (7,-.3){$a_7$};
\end{tikzpicture}
\end{center}

\textbf{Claim 4:}  $\dis(a_1,a_4)=1$.

\begin{proof}
By contradiction, let $\dis(a_1,a_4)=3$. 

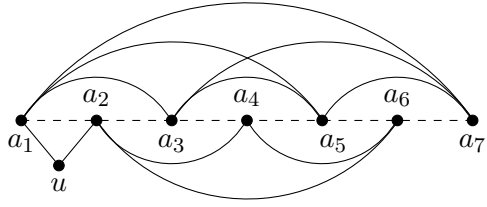
\begin{center}
\begin{tikzpicture}
\foreach \x in {1,2,...,7} \filldraw (\x,0) circle(2pt);
\draw [dashed](1,0)--(7,0);
 \draw (1,0) arc (150:30:1.15); 
\draw (1,0) arc (145:35:2.45);
 \draw (3,0) arc (150:30:1.15); 
 \draw (5,0) arc (150:30:1.15); 
\draw (3,0) arc (145:35:2.45);
\draw (1,0) arc (145:35:3.67);
 \draw (2,0) arc (-150:-30:1.15); 
  \draw (4,0) arc (-150:-30:1.15); 
\draw (2,0) arc (-145:-35:2.45);
\node at (1,-.3){$a_1$};\node at (2,.3){$a_2$};\node at (3,-.3){$a_3$};
\node at (4,.3){$a_4$};\node at (5.15,-.3){$a_5$};\node at (6,.3){$a_6$};
\node at (7,-.3){$a_7$};
\draw(1,0)--(1.5,-.6)--(2,0);\node at (1.5,-.85){$u$};
\filldraw (1.5,-.6) circle(2pt);
\end{tikzpicture}
\end{center}

If $\dis(u,a_6)=1$ then $\dis(a_1,a_6)\leqslant2$. Thus $\dis(a_1,a_6)=1$ and  $\dis(a_1,a_4)\leqslant2$, a contradiction. Therefore, $\dis(u,a_6)=2$.
If $\dis(u,a_7)=2$ then $a_7ua_7$  is a path in $D_2(G)$, a contradiction. Thus  $\dis(u,a_7)=1$.
Thus $\dis(a_2,a_7)\leqslant2$ and $\dis(a_2,a_7)=1$. Then, $\dis(a_4,a_7)\leqslant2$ and so $\dis(a_4,a_7)=1$. Which shows $\dis(a_1,a_4)\leqslant2$, a contradiction.
\end{proof}

Hence, $\dis(a_1,a_4)=1$, and by symmetry, $\dis(a_4,a_7)=1$  and then $\dis(a_1,a_6)=\dis(a_2,a_7)=\dis(a_2,a_5)=\dis(a_3,a_6)=1$. Thus, the induced subgraph generated by $A=\{a_1,a_2,\cdots,a_7\}$ is the following.

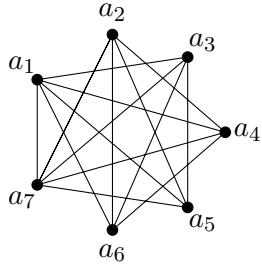
\begin{center}
\begin{tikzpicture}

\filldraw (0,.7) circle(2pt);\node at (-.2,.9){$a_1$};
\filldraw (1,1.3) circle(2pt);\node at (1,1.6){$a_2$};
\filldraw (2,1) circle(2pt);\node at(2.2,1.2){$a_3$};
\filldraw (2.5,0) circle(2pt);\node at (2.8,0){$a_4$};
\filldraw (2,-1) circle(2pt);\node at (2.2,-1.2){$a_5$};
\filldraw (1,-1.3) circle(2pt);\node at (1,-1.6){$a_6$};
\filldraw (0,-.7) circle(2pt);\node at (-.2,-.9){$a_7$};
\draw (0,.7)--(2,1)--(2,-1)--(0,-.7)--(1,1.3)--(2.5,0)--(1,-1.3)--(1,1.3)--(0,-.7)(2,1.)--(1,-1.3)--(0,.7)--(2.5,0)--(0,-.7)--(0,.7)--(2,-1);
\draw (0,-.7)--(2,1);\draw (1,1.3)--(2,-1);
\end{tikzpicture}
\end{center}

The diameter of this subgraph is 2 and since $\diam(G)=3$, $G$ has a vertex $b$ where $\dis(b,A)=1$. 

\textbf{Claim 5:}   $|N(b)\cap A|\geqslant4$.
\begin{proof}
By contradiction, suppose that $|N(b)\cap A|\leqslant3$. There are at least 4 vertices $a_{i_1},a_{i_2},a_{i_3},a_{i_4}$, where $1\leqslant i_1\leqslant  \cdots\leqslant i_4\leqslant 7$ such that $\dis(b,a_{i_k})=2$, $k=1,2,3,4$. 
Then $a_{i_1}ba_{i_4}$ is a path in $D_2(G)$ while $\diss(a_{i_1},a_{i_4})\geqslant3$, a contradiction.
\end{proof}

Let $B= N[A] $ and $u,v\in B\setminus A$. By Claim 5, $|N(u)\cap A|\geqslant4$ and $|N(v)\cap A|\geqslant4$ and then $|N(u)\cap N(v)|\geqslant1$ and so $\dis(u,v)\leqslant2$. Therfore $\diam(\langle B\rangle)=2$. Since $\diam(G)=3$, $G$ has a vertex $c\in N[B]\setminus B$. The vertex $c$ has an adjacent vertex $w\in B$ and since $|N(w)\cap A|\geqslant4$, there are at least $4$ vertices $a_{i_1},a_{i_2},a_{i_3},a_{i_4}$, 
where $1\leqslant i_1\leqslant  \cdots\leqslant i_4\leqslant 7$ such that $\dis(c,a_{i_k})=2$, $k=1,2,3,4$. 
Then $a_{i_1}ba_{i_4}$ is a path in $D_2(G)$, a contradiction.
\end{proof}

By Theorems \ref{the23} and \ref{the26} we have the following.

\begin{corollary}
If $\diam(G)=3$ then $2\leqslant \diam(D_2(G))\leqslant 5$ and the inequalities are sharp.
\end{corollary}
\begin{proof}
By Example \ref{exa26} the first inequality is sharp. Now, the following graph $G$ shows that the second inequality is sharp too.

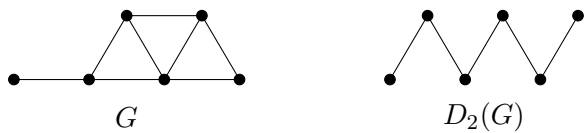
\begin{center}

\begin{tikzpicture}
\draw (0,0)--(3,0)--(2.5,.85)--(1.5,.85)--(1,0);
\draw (2.5,.85)--(2,0)--(1.5,.85);
\filldraw (0,0) circle(2pt);\filldraw (1,0) circle(2pt);\filldraw (2,0) circle(2pt);
\filldraw (3,0) circle(2pt);\filldraw (2.5,.85) circle(2pt);\filldraw (1.5,.85) circle(2pt);
\node at (1.5,-.5){$G$};

\draw (5,0)--(5.5,.85)--(6,0)--(6.5,.85)--(7,0)--(7.5,.85);
\filldraw (5,0) circle(2pt);\filldraw (6,0) circle(2pt);\filldraw (7,0) circle(2pt);
\filldraw (5.5,.85) circle(2pt);\filldraw (6.5,.85) circle(2pt);\filldraw (7.5,.85) circle(2pt);
\node at (6.25,-.5){$D_2(G)$};

\end{tikzpicture}
\end{center}
\end{proof}

\section{Conclusions}\label{sec3}

In this paper, we focused on the diameter of the connected $2$-distance graph of graphs. We found a sharp lower bound for  the diameter of $2$-distance graphs relative to the diameter of their initial graphs. Also, we showed  that $\diam(D_2(G))$ can be any integer $t\geqslant2$ when $\diam(G)=2$. But for $\diam(G)=3$, we obtained a sharp upper bound for $\diam(D_2(G))$. The above achievements raise the following questions.

\begin{conjecture}
Let $\diam(G)=n\geqslant3$. Then $\diam(D_2(G))\leqslant n+2$.
\end{conjecture}
\noindent
\textbf{Problem 3.2.} Charactrize all graphs $G$ which $\diam(G)=2$ and $\diam(D_2(G))=2 \text{ or }3$.

 In Theorem \ref{the23}, we proved $\diam(D_2(G))\geqslant \left\lceil \frac{1}{2}\diam(G)\right\rceil$ and in Example \ref{exa27} we show this inequality is sharp.

\noindent
\textbf{Problem 3.3.} Are there some graphs with $\diam(G)=n$ and $\diam(D_2(G))= \left\lceil\frac{1}{2}\diam(G)\right\rceil+1$.

\noindent
\textbf{\Large Declarations}

\noindent
\textbf{Conflict of interest} No potential conflict of interest was reported by the authors.

\end{document}